\documentclass[11pt]{amsart}

\numberwithin{equation}{section}
\usepackage{amssymb,amsmath}
\usepackage{amsfonts,amsthm,mathrsfs}
\usepackage{color,graphics,epsfig}

\newcommand{\mC}{\mathbb{C}}
\newcommand{\mD}{\mathbb{D}}

\newcommand{\mF}{\mathbb{F}}

\newcommand{\mN}{\mathbb{N}}

\newcommand{\mR}{\mathbb{R}}
\newcommand{\mS}{\mathbb{S}}

\newtheorem{theorem}{Theorem}[section]

\newtheorem{proposition}[theorem]{Proposition}

\theoremstyle{definition}

\theoremstyle{definition}
\newtheorem{definition}[theorem]{Definition}

\theoremstyle{definition}

\begin{document}



\title[Projective freeness of  real symmetric algebras]{Projective
  freeness of algebras of real symmetric functions}

\author{Amol Sasane}
\address{Department of Mathematics, Royal Institute of Technology,
   Stockholm 100 44, Sweden.}
\email{sasane@math.kth.se}

\begin{abstract}
  Let $\mD^n:=\{z=(z_1,\dots, z_n)\in \mC^n:|z_j|< 1, \;j=1,\dots
  ,n\}$, and let $\overline{\mD}^n$ denote its closure in $\mC^n$.
  Consider the ring
$$
C_{\textrm{r}}(\overline{\mD}^n;\mC)
=\{f:\overline{\mD}^n\rightarrow \mC:f\textrm{ is continuous and }
f(z)=\overline{f(\overline{z})} \;(z\in \overline{\mD}^n)\}
$$
with pointwise operations, where $\overline{\cdot}$ is used
appropriately to denote both (componentwise) complex conjugation and
closure.  It is shown that $ C_{\textrm{r}}(\overline{\mD}^n;\mC)$ is
projective free, that is, finitely generated projective modules over
$C_{\textrm{r}}(\overline{\mD}^n; \mC)$ are free. Let $A$ denote the
polydisc algebra, that is, the set of all continuous functions defined
on $\overline{\mD}^n$ that are holomorphic in $\mD^n$.  For $N$ a
positive integer, let $\partial^{-N} A$ denote the algebra of
functions $f\in A$ whose complex partial derivatives of all orders up
to $N$ belong to $A$.  We show the projective freeness of each of the
real symmetric algebras $
\partial^{-N}A_{\textrm{r}}
=\{f \in \partial^{-N} A:
f(z)=\overline{f(\overline{z})} \;(z\in \overline{\mD}^n)\}$.
\end{abstract}

\maketitle

\section{Introduction}

In this article we will show the projective freeness of some real
Banach algebras of real symmetric functions on the polydisc
$$
\overline{\mD}^n:=\{z=(z_1,\dots, z_n)\in \mC^n:|z_j|\leq 1,
  \;j=1,\dots ,n\}.
$$
We begin with the definition of a projective free
ring.

\begin{definition}
  Let $R$ be a commutative ring with identity. The ring $R$ is said to
  be {\em projective free} if every finitely generated projective
  $R$-module is free. Recall that if $M$ is a finitely generated
  $R$-module, then
\begin{enumerate}
\item $M$ is called {\em free} if $M \simeq R^k$ for some integer $k\geq 0$;
\item $M$ is called {\em projective} if there exists an $R$-module $N$
  and an integer $m\geq 0$ such that $M\oplus  N = R^m$.
\end{enumerate}
\end{definition}

\noindent In terms of matrices (see \cite[Proposition~2.6]{Coh} or
\cite[Lemma~2.2]{BalRodSpi}), the ring $R$ is projective free if and
only if every idempotent matrix $P$ is conjugate (by an invertible
matrix $S$) to a diagonal matrix with $1$s and $0$s on the diagonal,
that is, for every $m\in \mN$ and every $P\in R^{m\times m}$
satisfying $P^2=P$, there exists an $S\in R^{m\times m}$ such that $S$
is invertible as an element of $R^{m\times m}$ and
$$
S^{-1} P S=\left[\begin{array}{cc}
I_k & 0\\
0 & 0
\end{array}\right]
$$
for some $k\geq 0$.

It was shown in 1976 by Quillen and Suslin, independently, that the polynomial
ring $\mF[x_1, \dots , x_n]$ with coefficients in a field $\mF$ is
projective free, settling Serre's conjecture from 1955 (see \cite{Lam}).

In the context of a commutative semisimple unital complex Banach
algebra $R$, \cite[Corollary~1.4]{BruSas} says that the
contractability of the maximal ideal space $M(R)$ of $R$, equipped
with the weak-$\ast$ topology, suffices for $R$ to be projective free.
In particular, the complex Banach algebra $C(\overline{\mD}^n;\mC)$ of
complex-valued continuous functions on the polydisc $\overline{\mD}^n$
is projective free, since its maximal ideal space can be identified
with $\overline{\mD}^n$, which is clearly contractible. Similarly, the
polydisc algebra $A$ (the set of all continuous functions defined on
$\overline{\mD}^n$ that are holomorphic in $\mD^n$, with the usual
pointwise operations and the supremum norm) is also projective free.

Projective freeness of rings also play a role in Control Theory in the
context of stabilization of control systems over rings, see for
instance \cite{KhaSon}. The motivation for studying ``systems over
rings'' comes from situations in which one is interested in the study
of parametrized classes of linear control systems; this may be
handled by studying linear control systems in which the usual matrices
$A,B$ (describing the state differential equation $x'(t)=Ax(t)+Bu(t)$)
have entries that are functions of the parameters, with these
functions having a specific structure (polynomial, continuous,
analytic, etc.).  The solution to a control theoretic synthesis
problem over the ring will provide a parametrized family of solutions
to the corresponding problem for each system in the family; see
\cite{Son}. The interest in projective free rings stems from the fact
that many simplifications occur in the theory of systems over rings
under the assumption of projective freeness; for example in the
context of the problem of pole assignment; see
\cite[Corollary~3.6]{BSSV}. Moreover, in Control Theory, rings of real
symmetric functions in play a particularly important role. A function
$f: \mD^n\rightarrow \mC$ is said to be {\em real symmetric} if
$$
f(z)=\overline{f(\overline{z})} \quad (z\in \mD^n).
$$
(For $z=(z_1,\dots, z_n)\in \mC^n$,
$\overline{z}:=(\overline{z_1},\dots,\overline{z_n})$.)
Algebraic-analytic properties (some of which are relevant in control
theory) of various real symmetric algebras have also been studied
recently for example in \cite{MorRup}, \cite{MorWic0}, \cite{MorWic},
\cite{Wic2}.

It is known that the projective freeness of the complexification
$R\otimes \mC$ of an algebra $R$ over the reals does not in general
imply the projective freeness of $R$. For example, the ring
$C(\mS^2;\mC)$ of {\em complex}-valued continuous functions on the
unit sphere $\mS^2$ (in $\mR^3$) can be shown to be projective free
(see for example \cite[p.  38]{Lam}). On the other hand, a consequence
of the Hairy Ball Theorem (see for example \cite{Mil}) is that the
ring $C(\mS^2;\mR)$ of {\em real}-valued continuous functions on
$\mS^2$ is not projective free (see for example \cite[p. 30-33]{Lam}).

In this article, our first main result
is that the ring $C_{\textrm{r}}(\overline{\mD}^n;\mC)$ of real
symmetric functions from $C(\overline{\mD}^n;\mC)$, that is,
$$
C_{\textrm{r}}(\overline{\mD}^n;\mC)=\{f\in C(\overline{\mD}^n;\mC):
f(z)=\overline{f(\overline{z})} \;(z\in \overline{\mD}^n)\}
$$
is projective free. We note that $C_{\textrm{r}}(\overline{\mD}^n;\mC)$ is a
{\em real} Banach algebra, and not a complex Banach algebra with the
usual operations (for example we have $1\in C_{\textrm{r}}(\overline{\mD}^n;\mC)
\not\owns i\cdot 1$).

\begin{theorem}
\label{theorem}
$C_{\textrm{\em r}}(\overline{\mD}^n; \mC)$ is a projective free ring.
\end{theorem}

We will also show the projective freeness of certain subalgebras of
the real symmetric polydisc algebra. In order to state this result, we
will first introduce some notation.  For $N$ a positive integer, let
$\partial^{-N} A$ denote the Banach algebra of all functions $f\in A$
whose complex partial derivatives of all orders up to $N$ belong to
$A$, with the norm
$$
\|f\|_{\scriptscriptstyle \partial^{-N}
  A}:=\sum_{\alpha_1+\dots+\alpha_n\leq N}\frac{1}{\alpha_1 ! \dots
  \alpha_n!}
\sup_{z\in\mD^n}\left|\frac{\partial^{\alpha_1+\dots+\alpha_n}f}{\partial
    z_1^{\alpha_1} \dots \partial z_n^{\alpha_n}}(z)\right|.
$$
We show that each of the real symmetric algebras
$$
\partial^{-N}A_{\textrm{r}}
=\{f \in \partial^{-N} A:
f(z)=\overline{f(\overline{z})} \;(z\in \overline{\mD}^n)\}
$$
is projective free.

\begin{theorem}
\label{theorem2}
$\partial^{-N}A_{\textrm{\em r}}$ is a projective free ring.
\end{theorem}

We give the proofs of these two main results, namely
Theorem~\ref{theorem} and Theorem~\ref{theorem2}, in
Section~\ref{Section_theorem} and Section~\ref{Section_theorem2},
respectively.

\section{Proof of Theorem~\ref{theorem}}
\label{Section_theorem}

We will first recall the following result from vector bundle theory.

\begin{proposition}
\label{prop}
  $C([-1,1]^n; \mR)$ is a projective free ring.
\end{proposition}
\begin{proof} If $Q$ is a finitely generated projective module over
  $C([-1,1]^n; \mR)$, then it can be shown that there exists a vector
  bundle $\xi$ such that $Q\simeq S(\xi)$, where $S(\xi)$ denotes the
  $C([-1,1]^n; \mR)$-module consisting of all cross sections of $\xi$;
  see for example \cite[p.36, 3-F.(b)]{MilSta}. Moreover, $S(\xi)$ is
  free if and only if $\xi$ is trivial; see \cite[p.36,
  3-F.(a)]{MilSta}. But as $[-1,1]^n$ is contractible, $\xi$ is
  trivial (see for example \cite[Proposition~3.2.1, p.40]{Pot}).
\end{proof}

\begin{proof}[Proof of Theorem~\ref{theorem}]
We will prove by
  induction on $k$ ($0\leq k\leq n$) that each of the rings
  $C_{\textrm{r}}([-1,1]^{n-k}\times \overline{\mD}^k;\mC)$, that is,
$$
  \left\{f:[-1,1]^{n-k}\times \overline{\mD}^k\rightarrow \mC:
\begin{array}{ll}f\textrm{ is continuous and }\\
  f(z)=\overline{f(\overline{z})} \;(z \in [-1,1]^{n-k}\times
  \overline{\mD}^k)
\end{array}\right\},
$$
is projective free.

In order to avoid cumbersome notation, we set $
\Delta_k:=[-1,1]^{n-k}\times \overline{\mD}^k.$ Then
$[-1,1]^n=\Delta_0\subset \Delta_1\subset \dots \subset
\Delta_{n-1}\subset \Delta_n=\overline{\mD}^n$.

For $k=0$, the ring $C_{\textrm{r}}(\Delta_k;\mC)$ is precisely $C([-1,1]^n;
\mR)$, which is indeed projective free by Proposition~\ref{prop}.

Suppose that for some $k$ such that $0\leq k\leq n$, the ring $
C_{\textrm{r}}(\Delta_k;\mC)$ is projective free. We want to show that
$C_{\textrm{r}}(\Delta_{k+1};\mC)$ is then also projective free.

To this end, let us consider an idempotent $P$ with entries from
$C_{\textrm{r}}(\Delta_{k+1};\mC)$. In particular, the restriction
$P|_{\Delta_k}$ is an idempotent with all its entries belonging to
$C_{\textrm{r}}(\Delta_k;\mC)$.  By the induction hypothesis, since
$C_{\textrm{r}}(\Delta_k;\mC)$ is projective free, there is a $v\in
(C_{\textrm{r}}(\Delta_k;\mC) )^{m\times m}$ such that $v$ is invertible in
$(C_{\textrm{r}}(\Delta_k;\mC))^{m\times m}$, and
$$
v^{-1}(z)P(z) v(z)=\left[\begin{array}{cc} I_\ell & 0 \\0 &
    0\end{array}\right] \quad (z\in \Delta_k)
$$
for some $\ell\geq 0$.  Consider the matrix $\widetilde{P}$ given as
follows: for elements $z$ of
$$
\Delta_{k+1}^+:=\left\{ (r,a+ib,\zeta)\in \mC^n: \begin{array}{lll}
r\in [-1,1]^{n-k-1}\\
a\in \mR,\; b\geq 0,\; a+ib \in \overline{\mD},\\
\zeta \in \overline{\mD}^k\end{array}
\right\},
$$
we set
\begin{equation}
\label{eq_1}
\widetilde{P}(z)= (v(r,a,\zeta))^{-1} P(z) v(r,a,\zeta)\quad (z=(r,a+ib,\zeta) \in \Delta_{k+1}^+).
\end{equation}
Then for $z=(r,a+ib,\zeta) \in \Delta_{k+1}^+$, we have
\begin{eqnarray*}
(\widetilde{P}(z))^2
&=&
(v(r,a,\zeta))^{-1} P(z) v(r,a,\zeta)(v(r,a,\zeta))^{-1} P(z) v(r,a,\zeta)
\\
&=&(v(r,a,\zeta))^{-1} P(z) v(r,a,\zeta)=\widetilde{P}(z).
\end{eqnarray*}
Also, it is clear that $\widetilde{P}\in
(C(\Delta_{k+1}^+;\mC))^{m\times m}$. As $\Delta_{k+1}^+$ is
contractible, $C(\Delta_{k+1}^+;\mC)$ is projective free. Hence we can
find a $V \in (C(\Delta_{k+1}^+;\mC))^{m\times m}$ which is invertible
as an element of $(C(\Delta_{k+1}^+;\mC))^{m\times m}$ such that
\begin{equation}
\label{eq_2}
(V(z))^{-1}\widetilde{P}(z) V(z)=
\left[\begin{array}{cc} I_{\widetilde{\ell}} & 0 \\0 & 0\end{array}\right]\quad (z \in \Delta_{k+1}^+)
\end{equation}
for some $\widetilde{\ell}\geq 0$. As the pointwise rank of $P$
matches with that of $\widetilde{P}$ on $\Delta_{k+1}^+$, it follows
from \eqref{eq_1} and \eqref{eq_2} that $\ell=\widetilde{\ell}$.  Thus
\eqref{eq_2} becomes
$$
(V(z))^{-1}\widetilde{P}(z) V(z)=
\left[\begin{array}{cc} I_{\ell} & 0 \\0 & 0\end{array}\right] \quad (z \in \Delta_{k+1}^+).
$$
In particular, we have for $z=(r,a,\zeta) \in \Delta_k$ that
\begin{equation}
\label{eq_3}
(V(r,a,\zeta))^{-1}\widetilde{P}(r,a,\zeta) V(r,a,\zeta)=
\left[\begin{array}{cc} I_{\ell} & 0 \\0 & 0\end{array}\right].
\end{equation}
But
$$
\widetilde{P}(r,a,\zeta)=(v(r,a,\zeta))^{-1} P(r,a,\zeta) v(r,a,\zeta)=
\left[\begin{array}{cc} I_\ell & 0 \\0 & 0\end{array}\right].
$$
Hence from \eqref{eq_3} and the above, we have
$$
(V(r,a,\zeta))^{-1}\left[\begin{array}{cc} I_\ell & 0 \\0 & 0\end{array}\right]V(r,a,\zeta)
=\left[\begin{array}{cc} I_{\ell} & 0 \\0 & 0\end{array}\right] \quad ((r,a,\zeta) \in \Delta_k).
$$
By pre- and post- multiplying this with $V(r,a,\zeta)$ and $(V(r,a,\zeta))^{-1}$,
respectively, we obtain
$$
\left[\begin{array}{cc} I_\ell & 0 \\0 & 0\end{array}\right]
=V(r,a,\zeta)\left[\begin{array}{cc} I_{\ell} & 0 \\0 & 0\end{array}\right](V(r,a,\zeta))^{-1} \quad ((r,a,\zeta) \in \Delta_k).
$$
Now define $U$ by
$$
U(z)=v(r,a,\zeta) V(z) (V(r,a,\zeta))^{-1} \quad (z=(r,a+ib,\zeta)\in \Delta_{k+1}^+).
$$
First of all, $U\in (C(\Delta_{k+1}^+;\mC))^{m\times m}$ and
moreover, it is invertible as an element of
$(C(\Delta_{k+1}^+;\mC))^{m\times m}$.  We also have for
$z=(r,a+ib,\zeta)\in \Delta_{k+1}^+$ that
\begin{eqnarray}
\nonumber
& & (U(z))^{-1}P(z)U(z)\phantom{\left[\begin{array}{cc} I_\ell & 0 \\0 & 0\end{array}\right]}\\
\nonumber
&=&
V(r,a,\zeta) (V(z))^{-1} (v(r,a,\zeta))^{-1} P(z) v(r,a,\zeta) V(z)(V(r,a,\zeta))^{-1}\phantom{\left[\begin{array}{cc} I_\ell & 0 \\0 & 0\end{array}\right]}
\\\nonumber
&=&
V(r,a,\zeta) (V(z))^{-1}\widetilde{P}(z) V(z) (V(r,a,\zeta))^{-1}\phantom{\left[\begin{array}{cc} I_\ell & 0 \\0 & 0\end{array}\right]}\\
&=&
V(r,a,\zeta)\left[\begin{array}{cc} I_{\ell} & 0 \\0 & 0\end{array}\right](V(r,a,\zeta))^{-1}\label{eq_4}
=\left[\begin{array}{cc} I_\ell & 0 \\0 & 0\end{array}\right].
\end{eqnarray}
Finally, for $(r,a,\zeta)\in \Delta_k$,
 $
U(r,a,\zeta)=v(r,a,\zeta)
V(r,a,\zeta) (V(r,a,\zeta))^{-1}=v(r,a,\zeta),
$
and so $U|_{\Delta_k}$ belongs to $ (C_{\textrm{r}}(\Delta_k;\mC
))^{m\times m}$. For $z=(r,a+ib,\zeta)\in \Delta_{k+1}$, we set
$$
S(z)=S(r,a+ib,\zeta)=\left\{ \begin{array}{ll} U(r,a+ib,\zeta) & \textrm{if }b\geq 0\\
    \overline{U(r,a-ib,\overline{\zeta} )} & \textrm{if }b\leq 0.
\end{array}\right.
$$
(Here the notation $\overline{U(\xi)}$ is used to denote the matrix
obtained from $U(\xi)$ by taking complex conjugates entry-wise.)
Owing to the fact that $U|_{\Delta_k}$ belongs to $
(C_{\textrm{r}}(\Delta_k;\mC ))^{m\times m}$, the above gives us a
well-defined element $S$ in $ (C_{\textrm{r}}(\Delta_{k+1};
\mC))^{m\times m}$. Moreover, $S$ is invertible in
$(C_{\textrm{r}}(\Delta_{k+1}; \mC))^{m\times m}$.  Using the
real symmetry of $P$ and \eqref{eq_4}, we obtain that
$$
(S(z))^{-1} P(z) S(z)=
\left[\begin{array}{cc} I_\ell & 0 \\0 & 0\end{array}\right]
\quad (z\in \Delta_{k+1}).
$$
Thus $C_{\textrm{r}}(\Delta_{k+1}; \mC)$ is projective free. By
induction, it follows that $C_{\textrm{r}}(\Delta_{k}; \mC)$ is
projective free for all $0\leq k \leq n$. In particular, for $k=n$,
the ring $C_{\textrm{r}}(\Delta_{n};
\mC)=C_{\textrm{r}}(\overline{\mD}^n; \mC)$ is projective free.  This
completes the proof.
\end{proof}

\section{Proof of Theorem~\ref{theorem2}}
\label{Section_theorem2}

We begin by showing that the maximal ideal space of $\partial^{-N}A$
can be identified with the polydisc $\overline{\mD}^n$. This is a
generalization of an analogous one variable result; see \cite[Theorem~1.3]{SasTre}.

\begin{proposition}\label{prop_max_id_sp}
  The maximal ideal space of $\partial^{-N} A$ is homeomorphic to
  $\overline{\mD}^n$.
\end{proposition}
\begin{proof} For each $\lambda \in \overline{\mD}^n$, point
  evaluation at $\lambda$ is a complex homomorphism $\varphi_\lambda$.
  Moreover $\partial^{-N} A$ contains the functions $z_1,\dots, z_n$
  and so we see that if $\lambda\neq \mu$, the corresponding point
  evaluations $\varphi_\lambda$ and $\varphi_\mu$ are distinct. Thus
  the map $\lambda \mapsto \varphi_\lambda $ embeds $
  \overline{\mD}^n$ in the maximal ideal space of $\partial^{-N} A$.
  Also it is easy to see that this inclusion is continuous, since if
  $(\lambda_\alpha)_{\alpha\in I}$ is a net in $\overline{\mD}^n$
  which is convergent to $\lambda \in \overline{\mD}^n$, then for each
  function $f\in \partial^{-N} A$, we have that
 $
\varphi_{\lambda_\alpha}(f)= f(\lambda_\alpha)\longrightarrow f(\lambda)=\varphi_\lambda(f),
$
by the continuity of $f$, and so
$(\varphi_{\lambda_\alpha})_{\alpha\in I}$ converges to
$\varphi_\lambda$ in the weak-$\ast$ topology in the maximal ideal
space of $\partial^{-N} A$.  Now we will show that every complex
homomorphism arises in this manner.

  Let $\varphi$ be a nontrivial multiplicative linear functional on
  $\partial^{-N} A$, and let $\lambda:=(\varphi(z_1),\dots,
  \varphi(z_n))$. Then clearly for every polynomial $p$, $\varphi(p)=p(\lambda)$. We show below  that  for all $f\in \partial^{-N} A$,
\begin{equation}
\label{eq_1.6}
|\varphi(f)|\leq \|f\|_\infty.
\end{equation}
On applying this estimate to the coordinate functions $z\mapsto
z_\ell$, we see that $\lambda \in \overline{\mD}^n$. Since
$\partial^{-N} A\subset A$, any function $f\in \partial^{-N} A$ can be
approximated by polynomials in the $L^\infty$-norm. But \eqref{eq_1.6}
implies that $\varphi$ is continuous in the $L^\infty$-norm, and so
$\varphi(f)=f(\lambda)$ holds for all $f\in \partial^{-N} A$. (Note that in
this argument, we do not need the density of polynomials in the norm
of $\partial^{-N} A$, which doesn't hold for $N\geq 1$!)

Now we will prove \eqref{eq_1.6}. If $f\in \partial^{-N} A$ is such
that
$$
\inf_{z\in \mD^n} |f(z)| > 0,
$$
then $f$ is invertible in $\partial^{-N} A $.  Indeed, as
$\partial^{-N} A\subset A$, $ \displaystyle \inf_{z\in \mD^n} |f(z)| >
0 $ implies that $f$ is invertible in $A$.  Differentiating $1/f$ $N$
times we get that all its complex partial derivatives up to the order
$N$ are in $A$.  Therefore, if $0\not\in \textrm{clos.range}(f) =
\textrm{range}(f)$, then $f$ is invertible in $\partial^{-N} A$, and
so $f$ does not belong to any proper ideal of $\partial^{-N} A$. Thus
$\varphi(f)\neq 0$ for any maximal ideal (multiplicative linear
functional) $\varphi$.  Replacing $f$ by $f- a$, $a \in \mC$, we get
that if $a\not\in \textrm{range}(f)$, then for any multiplicative
linear functional $\varphi$, $\varphi(f) \neq a$, that is, $\varphi(f)
\subset \textrm{range}(f)$. Thus $|\varphi(f)| \leq \|f\|_\infty$, and
\eqref{eq_1.6} is proved.

We have seen that the Gelfand topology of the maximal ideal space of
$\partial^{-N} A$ is weaker than the usual Euclidean topology of
$\overline{\mD}^n$, and moreover it is Hausdorff. Also with the
Euclidean topology, $\overline{\mD}^n$ is compact. Consequently, the
two topologies coincide; see for example \cite[\S 3.8.(a)]{Rud}.
\end{proof}

We now prove Theorem~\ref{theorem2} by following the same method as in
the proof of the main result in \cite{Tan}.

\begin{proof}[Proof of Theorem~\ref{theorem2}]
  If $M$ is a finitely generated projective
  $\partial^{-N}A_{\textrm{r}}$-module, then $M\otimes \mC$ is a
  finitely generated projective $\partial^{-N} A$ ($\simeq
  \partial^{-N}A_{\textrm{r}}\otimes \mC$)-module. But by
  Proposition~\ref{prop_max_id_sp}, the maximal ideal space of
  $\partial^{-N}A$ is $\overline{\mD}^n$, which is contractible. So by
  \cite{BruSas}, it follows that $\partial^{-N}A$ is projective free,
  and so $M\otimes \mC$ is free as a $\partial^{-N} A$-module. Using
  this, we will show that $M$ is free as a $\partial^{-N}
  A_{\textrm{r}}$-module.

  Let $e_1,\dots, e_k$ be a basis for $M\otimes \mC$ over
  $\partial^{-N}A$. The involution $\cdot^*$ on $\partial^{-N}A$
  defined by $ f^*(z)=\overline{f(\overline{z})}$ ($z\in \mD^n$)
  induces an involution on $M\otimes \mC$, which by slight abuse of
  notation, we will also denote by $\cdot^*$: thus if $u+iv \in
  M\otimes \mC$, where $u,v\in M$, then $(u+iv)^*=u-iv$. It is clear
  that $e_1^*,\dots, e_k^*$ also form a basis for $M\otimes \mC$ over
  $\partial^{-N}A$.

Let $U=[u_{ij}]\in (\partial^{-N}A)^{k\times k}$ be the change of basis matrix
from the basis $\{e_1,\dots, e_k\}$ to the basis $\{e_1^*,\dots,
e_k^*\}$.

Then $ e_m^*=\displaystyle \sum_{\ell=1}^k u_{m\ell} e_\ell, $
$\displaystyle e_m=e_m^{**}= \sum_{\ell=1}^k u_{m\ell}^* e_\ell^*$,
and so $ e_m^*=\displaystyle \sum_{\ell=1}^k \sum_{j=1}^k
u_{m\ell}u_{\ell j}^* e_j^*.$ Consequently,
\begin{equation}
\label{eq_Us_relation}
U(z)\overline{U(\overline{z})}=I \quad(z\in \overline{\mD}^n).
\end{equation}
Suppose we are able to find an matrix $V\in (\partial^{-N}A)^{k\times
  k}$ satisfying
\begin{equation}
\label{eq_fact_of_U}
U(z)=V(z) (\overline{V(\overline{z})})^{-1} \quad(z\in \overline{\mD}^n).
\end{equation}
Let us first observe that this would give the desired freeness of the
$\partial^{-N}A_{\textrm{r}}$-module $M$. To this end, let us define
$$
\left[ \begin{array}{ccc}
 \widetilde{e}_1 \\ \vdots \\ \widetilde{e}_k
\end{array}\right]
=
(\overline{V(\overline{z})})^{-1}
\left[ \begin{array}{ccc}
 e_1 \\ \vdots \\ e_k\end{array}\right]
+
(V(z))^{-1} \left[ \begin{array}{ccc}
 e_1^* \\ \vdots \\ e_k^*\end{array}\right] .
$$
We note first of all that $\widetilde{e}_m=\widetilde{e}_m^*$ and so
$\widetilde{e}_m\in M$. Next we claim that
$\widetilde{e}_1,\dots,\widetilde{e}_k$ are linearly independent over
$\partial^{-N}A_{\textrm{r}}$. Indeed, if there exist elements
$\alpha_1,\dots, \alpha_k\in \partial^{-N}A_{\textrm{r}}$ such that
$$
\sum_{m=1}^k \alpha_m \widetilde{e}_m=0,\textrm{
then }
0=
2 \left[ \begin{array}{ccc} \alpha_1 & \dots & \alpha_k
  \end{array}\right]
(\overline{V(\overline{z})})^{-1}
\left[
  \begin{array}{ccc} e_1 \\ \vdots \\ e_k\end{array}\right],
$$
which implies that $\alpha_1=\dots= \alpha_k=0$ (by the independence
of $e_1,\dots, e_k$ over $\partial^{-N}A$). Also, as
$$
\left[ \begin{array}{ccc}
 e_1 \\ \vdots \\ e_k\end{array}\right]
=
\frac{1}{2}\overline{V(\overline{z})} \left[ \begin{array}{ccc}
 \widetilde{e}_1 \\ \vdots \\ \widetilde{e}_k\end{array}\right] ,
$$
it follows that the span of $\widetilde{e}_1,\dots,\widetilde{e}_k$
over the ring $\partial^{-N}A_{\textrm{r}}$ is $M$. This completes the proof
that once we have a $V\in (\partial^{-N}A)^{k\times k}$ satisfying
\eqref{eq_fact_of_U}, then we have the desired freeness of $M$. Now we
will give the construction of the required $V$. In order to do this, we
will have to first deform $U$. Set
\begin{equation}
\label{eq_fact_of_U_t}
U_t(z):= U(tz) \quad (0\leq t \leq 1, \; z\in \overline{\mD}^n).
\end{equation}
This defines a path from $U(0)$ to $U(z)$. We claim that we can find
the needed $V(z)$ along this path. Hence we want $ U_t(z)= V_t(z)
(\overline{V_t(\overline{z})})^{-1}.  $ Differentiating both sides
with respect to $t$ (with a fixed $z$) and post-multiplying by
$(U_t(z))^{-1}$ gives us
\begin{equation}
\label{eq_complicated_DE}
\frac{\partial U_t}{\partial t} U_t^{-1}
=
\frac{\partial V_t}{\partial t} V_t^{-1}
-
U_t \overline{\frac{\partial V_t}{\partial t} V_t^{-1}  (\overline{z})}U_t^{-1}.
\end{equation}
So we need to solve this differential equation.

Consider instead the equation
\begin{equation}
\label{eq_aux_DE}
\frac{\partial V_t}{\partial t} V_t^{-1}
=
\frac{1}{2} \frac{\partial U_t}{\partial t} U_t^{-1}
\end{equation}
with the initial condition $V_0(z)$ being any complex matrix such that
\begin{equation}
\label{eq_aux_init_cond}
U(0)= V_0(z)( \overline{V_0(z)})^{-1}.
\end{equation}
The existence of this matrix $V_0(z)$ will be shown later. For each
fixed $z$, the equation \eqref{eq_aux_DE} has a solution $t\mapsto V_t
(z)$, and by \cite[\S 2.3, p.59]{LinSeg}, we have that for each fixed
$t$, the map $z\mapsto V_t(z): \mD^n \rightarrow \mC^{k\times k}$ is
holomorphic. Also, since $(t,z) \mapsto U(tz): [0,1]\times
\overline{\mD}^n\rightarrow \mC^{k\times k}$ is $C^N$, it follows that
the right hand side of the equation \eqref{eq_aux_DE} is $C^{N-1}$,
and so the map $(t,z) \mapsto V_t(z): [0,1]\times
\overline{\mD}^n\rightarrow \mC^{k\times k}$ is $C^N$ as well. (In
particular, from the above observations, it follows that
$V:=V_1(\cdot)\in (\partial^{-N}A)^{k\times k}$.)

Now we show that the solution $V_t$ to our auxiliary differential
equation \eqref{eq_aux_DE} in fact solves the original differential
equation \eqref{eq_complicated_DE}. Note first that from
\eqref{eq_Us_relation}, we obtain
$$
U_t^{-1}\frac{\partial U_t}{\partial t} +\overline{ \frac{\partial U_t}{\partial t} U_t^{-1}(\overline{z})}=0.
$$
We have
$$
\frac{\partial V_t}{\partial t} V_t^{-1}
-
U_t \overline{\frac{\partial V_t}{\partial t} V_t^{-1}  (\overline{z})}U_t^{-1}
=
 \frac{1}{2} \frac{\partial U_t}{\partial t} U_t^{-1}
 -
 U_t \overline{\frac{1}{2} \frac{\partial U_t}{\partial t} U_t^{-1}(\overline{z})}U_t^{-1}
=
\frac{\partial U_t}{\partial t} U_t^{-1}.
$$
The only remaining step is to show the existence of a matrix $V_0(z)$
such that \eqref{eq_aux_init_cond} holds.  Let $U(0)=P+iQ$, where
$P,Q\in\mR^{k\times k}$. Using \eqref{eq_Us_relation}, it follows that
$P^2+Q^2=I$ and $PQ=QP$.  Since $U(0)$ is invertible, it has a
logarithm, given by the holomorphic function calculus:
$$
\log U(0) =\frac{1}{2\pi i} \int_\gamma (\log\zeta )(\zeta I-(P+iQ))^{-1} d\zeta,
$$
where $\gamma$ is any simple curve in the domain $\mC\setminus
e^{i\theta}(-\infty, 0]$ that contains the spectrum of $U(0)$ in its
interior, and the $\theta$ is chosen so that the ray
$e^{i\theta}(-\infty, 0]$ avoids all eigenvalues of $U(0)$.  Also,
since $PQ=QP$, it is clear that the real and imaginary parts of
$(\zeta I-(P+iQ))^{-1}$ commute with each other.  As $\log {\zeta}$ is
just a scalar, it follows that the real and imaginary parts of $\log
U(0)$ commute, that is, if $U(0)=X+iY$, with $X,Y\in\mR^{k\times k}$,
then $XY=YX$. In particular, $X+iY$ and $X-iY$ also commute. We have
$U(0)=e^{X+iY}$, and thus \eqref{eq_Us_relation} gives $
e^{2X}=e^{X+iY}e^{X-iY}=U(0)\overline{U(0)}=I.$ Now let $S$ be an
invertible matrix taking $X$ to its Jordan canonical form: $ X=SJ
S^{-1}$, where $J$ is a block diagonal matrix with Jordan blocks.
Since $e^{2X}=I$, we obtain $Se^{2J}S^{-1}=I$, and so $(e^{J})^2=I$.
But for a Jordan sub-block
$$
J_\ell:=\left[\begin{array}{cccc}
    \lambda & 1 & &  \\
    & \ddots &\ddots   & \\
    & & &  1\\
    & & & \lambda \end{array}\right]\in \mC^{m\times m}, \;\;
e^{J_\ell}=e^{\lambda}\left[\begin{array}{ccccc}
1 & \frac{1}{1!} & \frac{1}{2!} & \dots & \frac{1}{(m-1)!} \\
  & 1 & \frac{1}{1!} & \frac{1}{2!} & \\
  &   & \ddots & \ddots & \frac{1}{2!}\\
  &   &        & \ddots & \frac{1}{1!} \\
 & & & & 1
\end{array}\right],
$$
and so we conclude that all eigenvalues of $J$ must be $0$, and
moreover, each Jordan block $J_\ell$ is of size $1\times 1$. So $J=0$
and so $X=0$.  Define $V_0(z):=e^{iY/2}$. Then $(
\overline{V_0(z)})^{-1}=(e^{-iY/2})^{-1}=e^{iY/2}$. Consequently,
$$
 V_0(z)( \overline{V_0(z)})^{-1}=e^{iY/2}\cdot e^{iY/2}=e^{iY}=U(0),
$$
as required.
\end{proof}

In light of our result above, it is natural to ask the following
question: Is the real symmetric polydisc algebra
$A_{\textrm{r}}:=\{f\in A:f(z)=\overline{f(\overline{z})} \; (z\in
\mD^n)\}$ projective free?

% \medskip

% \noindent The real symmetric polydisc algebra
% $$
% A_{\textrm{r}}(\overline{\mD}^n):= C_{\textrm{r}}(\overline{\mD}^n;\mC)\cap H_{\textrm{r}}( \mD^n)
% $$
% and both $ C_{\textrm{r}}(\overline{\mD}^n;\mC)$ as well as $H_{\textrm{r}}( \mD^n)$ are
% projective free. Moreover, since the maximal ideal space of the
% polydisc algebra
% $$
% A(\overline{\mD}^n):= C(\overline{\mD}^n;\mC)\cap H( \mD^n)
% $$
% is the polydisc $\overline{\mD}^n$, which is contractible, it follows
% that $A(\overline{\mD}^n)$ is projective free.

\medskip

\noindent {\bf Acknowledgements:} The author thanks Rudolf Rupp for
finding an error in an earlier argument of Theorem~\ref{theorem}, and
the referee for useful remarks which improved the presentation of the
article. A discussion with Thomas Schick is also gratefully
acknowledged.

\end{document}